\documentclass{amsart}

\usepackage{latexsym,amssymb}
\usepackage{amsmath,amsthm}

\usepackage{amsfonts}
\usepackage{graphicx}   
\usepackage{url}

\def\diver{\mathop{\text{\normalfont div}}}
\def\cp{\operatorname{\text{cap}}}

\def\quactp{$(s,p)-$quasi everywhere}

\def\qe{$(s,p)-$q.e.}

\newcommand{\R}{\mathbb{R}}

\newcommand{\N}{\mathbb{N}}

\newcommand{\ve}{\varepsilon}
\newcommand{\ito}{\infty}

\newcommand{\A}{\mathcal{A}}

\newtheorem{teo}{Theorem}[section]
\newtheorem{coro}[teo]{Corollary}
\newtheorem{lem}[teo]{Lemma}
\newtheorem{pro}[teo]{Proposition}
	
\theoremstyle{definition}
\newtheorem{de}[teo]{Definition}
\theoremstyle{remark}
\newtheorem{rk}[teo]{Remark}

\numberwithin{equation}{section}

\begin{document}

\title[Continuity results with respect to domain perturbations]{Continuity results with respect to domain perturbation for the fractional $p-$laplacian}

\author[C. Baroncini, J. Fern\'andez Bonder and J.F. Spedaletti]{Carla Baroncini, Juli\'an Fern\'andez Bonder and Juan F. Spedaletti}

\address[C. Baroncini]{Departamento de Matem\'atica FCEN - Universidad de Buenos Aires and IMAS - CONICET. Ciudad Universitaria, Pabell\'on I (C1428EGA)
Av. Cantilo 2160. Buenos Aires, Argentina.}

\email{cbaronci@dm.uba.ar}

\address[J. Fern\'andez Bonder]{Departamento de Matem\'atica FCEN - Universidad de Buenos Aires and IMAS - CONICET. Ciudad Universitaria, Pabell\'on I (C1428EGA)
Av. Cantilo 2160. Buenos Aires, Argentina.}

\email{jfbonder@dm.uba.ar}

\urladdr{http://mate.dm.uba.ar/~jfbonder}

\address[J. F. Spedaletti]{Departamento de Matem\'atica, Universidad Nacional de San Luis and IMASL - CONICET. Ej\'ercito de los Andes 950 (D5700HHW), San Luis, Argentina.}

\email{jfspedaletti@unsl.edu.ar}

\subjclass[2010]{35B30, 35J60}

\keywords{Fractional $p-$laplacian, domain perturbation, fractional capacity}

\begin{abstract}
In this paper we give sufficient conditions on the approximating domains in order to obtain the continuity of solutions for the fractional $p-$laplacian. These conditions are given in terms of the fractional capacity of the approximating domains.
\end{abstract}

\maketitle

\section{Introduction.}

In recent years, there has been an increasing amount of attention in nonlocal problems due to some interesting new applications that these operators have shown to possess, such as some models for physics \cite{DGLZ, Eringen, Giacomin-Lebowitz, Laskin, Metzler-Klafter, Zhou-Du}, finances \cite{Akgiray-Booth, Levendorski, Schoutens}, fluid dynamics \cite{Constantin}, ecology \cite{Humphries, Massaccesi-Valdinoci, Reynolds-Rhodes} and image processing \cite{Gilboa-Osher}.

In particular, the so-called $(s,p)-$laplacian operator have been extensively studied and up to date is almost impossible to give an exhaustive list of references. See for instance \cite{DiNezza-Palatucci-Valdinoci, Demengel} and references therein.

The $(s,p)-$laplace operator is defined as
$$
(-\Delta_p)^s u (x) := 2\text{ p.v.} \int_{\R^n} \frac{|u(x)-u(y)|^{p-2}(u(x)-u(y))}{|x-y|^{n+sp}}\, dy,
$$
up to some normalization constant. The term p.v. stands for {\em principal value}.

It is easy to see that this operator is bounded between the fractional order Sobolev space $W^{s,p}(\R^n)$ and its dual $W^{-s,p'}(\R^n)$. Moreover, for any $u\in W^{s,p}(\R^n)$, $(-\Delta_p)^s u$ defines a distribution as
$$
\langle (-\Delta_p)^s u, \phi\rangle = \iint_{\R^n\times \R^n} \frac{|u(x)-u(y)|^{p-2}(u(x)-u(y))(\phi(x)-\phi(y))}{|x-y|^{n+sp}}\, dxdy,
$$
for every $\phi\in C^\infty_c(\R^n)$. In fact this equality holds for any $\phi\in W^{s,p}(\R^n)$. See next section for precise definitions of the Sobolev spaces $W^{s,p}(\R^n)$.

Another elementary fact is that given $f\in L^{p'}(\R^n)$ (or more generaly $f\in W^{-s,p'}(\R^n)$) and a bounded open set $\Omega\subset \R^n$ there exists a unique $u\in W^{s,p}_0(\Omega) = \{v\in W^{s,p}(\R^n)\colon v=0 \text{ a.e. in }\R^n\setminus\Omega\}$ that verifies
$$
\begin{cases}
(-\Delta_p)^s u = f & \text{in }\Omega\\
u = 0 & \text{in }\R^n\setminus \Omega,
\end{cases}
$$
where the equality is understood in the sense of distributions.

We denote this function by $u_\Omega^f$.

The question that we address in this paper is then the following. Assume that we have a sequence of domains $\{\Omega_k\}_{k\in\N}$ such that $\Omega_k\to \Omega$ in a suitable defined notion of convergence of sets. Is it then true that $u_{\Omega_k}^f\to u_\Omega^f$ in some sense? Or more generally, give necessary and/or sufficient conditions for the above statement to hold true. 

When the $(s,p)-$laplacian is replaced by the classical $p-$laplace operator $\Delta_p u:= \diver(|\nabla u|^{p-2}\nabla u)$ (recall that for $p=2$ this operator becomes the classical Laplace operator), this problem was studied in \cite{Sverak}. In that article, the author gives aditional conditions in terms of the capacity of the symmetric differences of the domains in order to obtain a positive answer, and the famous counterexample of Cioranescu and Murat \cite{C-M} says that one cannot expect a positive answer without any further assumptions.

In the fractional setting, recently \cite{C-M} extended the counterexample of Cioranescu-Murat to the $(s,p)-$laplacian so, as in the classical setting, one cannot expect a positive answer in full generality.

Therefore, our purpose in this work is to find some capacitary conditions on the symmetric diference $\Omega_k\triangle\Omega$ in order to have convergence of the solutions $u_{\Omega_k}^f\to u_\Omega^f$.

\subsection*{Organization of the paper}
After this introduction, in section 2 we revise the definitions and results on fractional order Sobolev spaces and on fractional capacities that are needed in the paper. Then, in section 3, we prove our main result (Theorem \ref{teo.main}).

\section{Preliminaries}
Let $\Omega \subset \R^n $ be an open, connected set. For $0 < s < 1 < p < \infty$, we consider the fractional order Sobolev space $W^{s,p}(\Omega)$ defined as follows

$$
W^{s,p}(\Omega) := \left\{ u\in L^p(\Omega) \colon \frac{|u(x)-u(y)|}{|x-y|^{\frac{n}{p}+s}} \in L^p(\Omega \times \Omega)\right\},
$$
endowed with the natural norm
$$
\|u\|_{W^{s,p}(\Omega)} = \|u\|_{s,p;\Omega} = \left(\int_{\Omega} |u|^p \, dx + \iint_{\Omega \times \Omega} \frac{|u(x)-u(y)|^p}{|x-y|^{n + sp}} \, dx dy \right)^\frac{1}{p}.
$$

The term

$$
[u]_{W^{s,p}(\Omega)}^p = [u]^p_{s,p; \Omega} = \iint_{\Omega \times \Omega} \frac{\left|u(x)-u(y)\right|^p}{\left|x-y\right|^{n+sp}} \, dx \, dy
$$

is called the \textit{Gagliardo seminorm} of $u$. We refer the interested reader to \cite{DiNezza-Palatucci-Valdinoci} for a throughout introduction to these spaces.

When $\Omega=\R^n$, we omit it in the notation, i.e.
$$
\|u\|_{s,p;\R^n} = \|u\|_{s,p} \text{ and } [u]_{s,p;\R^n} = [u]_{s,p}.
$$

In order to consider Dirichlet boundary conditions, it is customary to define the spaces 
$$
W_0^{s,p}(\Omega) := \{u\in W^{s,p}(\R^n)\colon u = 0 \text{ a.e. in } \R^n \setminus \Omega\}.
$$

Let us observe that $W_0^{s,p}(\Omega)$ is a closed subset of $W^{s,p}(\R^n)$. Therefore it has the same properties as a functional space. That is, $\Big(W_0^{s,p}(\Omega), \| \cdot \|_{s,p}\Big)$ is a separable, uniformly convex and reflexive Banach space.

An alternative definition for $W^{s,p}_0(\Omega)$ is to consider the closure of $C_c^\infty(\Omega)$ in $W^{s,p}(\R^n)$ with respect to the norm $\|\cdot\|_{s,p}$. If $\Omega$ is Lipschitz, both definitions are known to coincide (see \cite{DiNezza-Palatucci-Valdinoci}).

\subsection{Elementary properties}
We will now present some well-known properties of the norm that will be useful for our results. We state the results without proof for future references.

\begin{pro}[Poincar\'e Inequality]\label{poincare}
Let $\Omega\subset \R^n$ be an open set of finite measure. Then, there exists a positive constant $c = c(s,p,n,\left|\Omega\right|) > 0$ such that
$$
\|u\|_p \leq c[u]_{s,p}\ \text{for every } u \in W_0^{s,p}(\Omega).
$$
\end{pro}

\begin{coro}
Let $\Omega\subset \R^n$ be an open set of finite measure. Then, $[\,\cdot\,]_{s,p}$ and $\|\cdot\|_{s,p}$ define equivalent norms in $W_0^{s,p}(\Omega)$.
\end{coro}

We will now define a notion of convergence of domains that will be essential for our next results. 

\begin{de}[Hausdorff complementary topology.] Let $D\subset \R^n$ be compact. Given $K_1, K_2\subset D$ compact sets, we define de Hausdorff distance $d_H$ as
$$
d_H(K_1, K_2) := \max\left\{\sup_{x\in K_1} \inf_{y\in K_2} \|x-y\|, \sup_{x\in K_2} \inf_{y\in K_1} \|x-y\|\right\}.
$$

Now, let $\Omega_1,\Omega_2\subset D$ be open sets, we define the Hausdorff complementary distance $d^H$ as
$$
d^H(\Omega_1, \Omega_2) := d_H(D \setminus \Omega_1, D \setminus \Omega_2).
$$

Finally, we say that $\{\Omega_k\}_{k\in \N}$ converges to $\Omega$ in the sense of the Hausdorff complementary topology, denoted by $\Omega_k \stackrel{H}{\to}\Omega$, if $d^H(\Omega_k, \Omega)\to 0$.

We will use the notation
$$
\A(D) := \{\Omega\subset D\colon \text{ open}\}
$$
and therefore this space has a natural structure of a metric space with metric $d^H$.
\end{de}

\begin{rk}
Is a well known fact that the space $(\A(D), d^H)$ is a compact metric space when $D$ is compact.
\end{rk}

For the proof of the following proposition, we refer to the book \cite{Henrot}.

\begin{pro}\label{convH}
If $\Omega_k\stackrel{H}{\to}\Omega$, then for every $\phi \in C_c^\infty(\Omega)$ there is an integer $k_0$ such that $\phi\in C_c^\infty(\Omega_k)$ for $k\geq k_0$.
\end{pro}

\subsection{Fractional Capacity}

In this subsection, we recall some definitions of the $(s,p)-$capacity and relative capacity that can be found, for instance, in \cite{Warma}.

For a detailed analysis of the $(s,p)-$capacity, we refer to the above mentioned article \cite{Warma}.

We start with the definition of the $(s,p)-$capacity and the relative $(s,p)-$fractional capacity. 

\begin{de}
Let $E\subset\R^n$ be an arbitrary set. We define the {\em $(s,p)-$fractional capacity} of the set $E$ as
\begin{equation}
\begin{split}
\cp_{s,p}(E)&:= \inf \{[u]_{s,p}^p\colon u\in C_c^\ito(\R^n), u\geq 1\text{ in }E\}\\ 
\end{split}
\end{equation}
\end{de}

Given $\Omega\subset\R^n$ an open and bounded set and $E\subset \Omega$, we can define the capacity of the set $E$ relative to the set $\Omega$ as follows.
\begin{de}
$$
\cp_{s,p}(E;\Omega)=\inf \{[u]_{s,p}^p\colon  u\in W_0^{s,p}(\Omega), u\ge 1\text{ in an open neighborhood of }E\}.
$$
\end{de}
\begin{rk}
It is an immediate consequence of the above definitions that $\cp_{s,p}(E)\le \cp_{s,p}(E;\Omega)$.
\end{rk}

Now, when we deal with pointwise properties of Sobolev functions we must change the concept of {\em almost everywhere} for {\em quasi everywhere}. The following definition expresses such idea.
\begin{de}
We say that an property is valid \quactp\ if it is valid except in a set of null $(s,p)-$capacity. We note this fact writing \qe
\end{de}

\begin{de}
Let $D \subset \R^n$ be an open and bounded set, we say that $\Omega \subset D$ is {\em $(s,p)-$quasi open} if there is a decreasing sequence $\{\omega_k\}_{k\in\N}$ of open sets such that $\cp_{s,p}(\omega_k,D)\to 0$ and $\Omega\cup \omega_k$ is an open set for each $k\in\N$.
\end{de}

\begin{de}
A function $u \colon \Omega \to \R$ is called an {\em $(s,p)-$quasi continuous} function if for every $\ve>0$, there is an open set $U$ such that $\cp_{s,p}(U, \Omega)<\ve$ and $u|_{\Omega \setminus U}$ is continuous.
\end{de}

The next results, which proofs can be found in \cite{Warma} will be needed in the course of the proof of the main result of this paper.

\begin{teo}[Theorem 3.7 in \cite{Warma}]
For each $u\in W^{s,p}(\R^n)$ there exists a $(s,p)-$quasicontinuous function $v\in W^{s,p}(\R^n)$ such that $u=v$ a.e. in $\R^n$.
\end{teo}

\begin{rk}
It is easy to see  that two $(s,p)-$quasicontinuous representatives of a given function $u\in W^{s,p}(\R^n)$ can only differ in a set of zero $(s,p)-$capacity. Therefore, the unique $(s,p)-$quasicontinuous representative (defined \qe) of $u$ will be denoted by $\tilde{u}$.
\end{rk}

\begin{pro}[Lemma 3.8 in \cite{Warma}]\label{propsubsequence}
Let $0<s<1<p<\ito$. and let $\{v_k\}_{k\in \N}\subset W^{s,p}(\R^n)$ be such that $v_k\rightarrow v$ in $W^{s,p}(\R^n)$ for some $v\in W^{s,p}(\R^n)$. Then there is a subsequence $\{v_{k_j}\}_{j\in\N}\subset \{v_k\}_{k\in\N}$ such that $\tilde{v}_{k_j}\to \tilde{v}$ \qe
\end{pro}

\begin{teo}[Theorem 4.5 in \cite{Warma}]\label{teocaracterizacion}
Let $D\subset \R^n$ be an open set and $\Omega\subset D$ an open subset. Then,
$$
u\in W^{s,p}_0(\Omega) \Leftrightarrow u\in W^{s,p}_0(D)\text{ and } \tilde{u}=0  \text{ \qe\ in } D \setminus \Omega.
$$
\end{teo}


\section{Continuity of the problems with respect to variable domains}

Throughout this section we consider $0<s<1<p<\ito$ to be fixed.

Let $D\subset \R^n$ be a bounded, open set and let $\Omega\subset D$ be an open set. The Dirichlet problem for the $(s,p)-$laplacian consists of finding $u$ $\in W^{s,p}_{0}(\Omega)$ such that 
\begin{equation}\label{eq.dirichlet}
\begin{cases}
(-\Delta)_p^s u=f &\text{ in }\Omega,\\
u=0 & \text{ on } \Omega^c:= \R^n\setminus \Omega,
\end{cases}
\end{equation}
where $f\in W^{-s,p'}(D):= [W^{s,p}_0(D)]'$.

In its weak formulation, this problem consists of finding $u$ $\in W^{s,p}_{0}(\Omega)$ such that
$$
\langle (-\Delta)_p^s u, v\rangle=\langle f, v\rangle \text{ for every } v \in W^{s,p}_{0}(\Omega).
$$
That is, for every $v \in W^{s,p}_{0}(\Omega)$, the following equality holds
$$
\iint_{\R^n \times \R^n}\frac{\left|u(x)-u(y)\right|^{p-2}(u(x)-u(y))(v(x)-v(y))}{\left|x-y\right|^{n+sp}}\, dx\, dy = \langle f, v\rangle.
$$

\begin{lem}
Let $f\in W^{-s,p'}(D)$ and $\Omega \in \A(D)$. Then there exists a unique $u\in W^{s,p}_0(\Omega)$, which we will denote $u^{f}_{\Omega}$, solution of \eqref{eq.dirichlet}.
\end{lem}

\begin{proof}
It is enough to consider $\Im \colon W^{s,p}_{0}(\Omega) \to \R$ defined by $\Im(v):=\frac{1}{p}[v]^{p}_{s,p}-\langle f, v\rangle$ and observe that $u$ is solution of \eqref{eq.dirichlet} if and only if $u$ is a minimizer for $\Im$.
Since $\Im$ has a unique minimizer (observe that $\Im$ is strictly convex), this completes the proof.
\end{proof}

Now we observe that these solutions $u_\Omega^f$ are bounded independently of $\Omega$.
\begin{lem}\label{cota.solucion}
There is a constant $C=C(\|f\|_{-s,p'},s,p,n,|D|)$ such that $\|u_\Omega^f\|_{s,p}\leq C$ for every $\Omega \in \A(D)$.
\end{lem}

\begin{proof}
Let us observe that 
$$
[u^{f}_{\Omega}]_{s,p}^{p}=\langle (-\Delta)_p^s u^{f}_{\Omega}, u^{f}_{\Omega}\rangle=\langle f, u^{f}_{\Omega}\rangle\leq \|f\|_{-s,p'}\|u^{f}_{\Omega}\|_{s,p}.
$$
Combining this inequality with Theorem \ref{poincare}, there exists a constant $C>0$ such that 
$$
\|u^{f}_{\Omega}\|_{s,p}^{p}\leq C\|f\|_{-s,p'}\|u^{f}_{\Omega}\|_{s,p},
$$
from where the conclusion of the lemma follows.
\end{proof}

As an immediate corollary, we have the following result.

\begin{coro}\label{weakconvergence}
Let $\{\Omega_k\}_{k\in\N}\subset \A(D)$. Then, $\{u_{\Omega_k}^f\}_{k\in\N}$ is bounded in $W^{s,p}_{0}(D)$ and, therefore, there exists $u^{*} \in W^{s,p}_{0}(D)$ and a subsequence $\{u_{\Omega_{k_j}}^f\}_{j\in\N}\subset \{u_{\Omega_k}^f\}_{k\in\N}$ such that $u_{\Omega_{k_j}}^f \rightharpoonup u^{*}$ weakly in $W^{s,p}_{0}(D)$.
\end{coro}

The next result is a first step in proving the continuity result.
\begin{teo}\label{u*sol}
Let $\{\Omega_k\}_{k\in\N}\subset \A(D)$ and $\Omega\in \A(D)$ be such that $\Omega_k\stackrel{H}{\to} \Omega$. Assume that  $u_{\Omega_k}^f\rightharpoonup u^*$ weakly in $W^{s,p}_0(D)$ for some $u^*\in W^{s,p}_0(D)$ when $k\to \ito$. Then 
$$
(-\Delta)_p^s u^*=f \text{ in }\Omega,
$$  
in the sense of distributions. That is 
\begin{equation}\label{debilu*}
\iint_{\R^n \times \R^n}\frac{|u^*(x)-u^*(y)|^{p-2}(u^*(x)-u^*(y))(\phi(x)-\phi(y))}{|x-y|^{n+sp}}\,dx\,dy=\langle f,\phi\rangle,
\end{equation}
for every $\phi\in C_c^\ito(\Omega)$.
\end{teo}

\begin{proof}
We denote $u_k = u_{\Omega_k}^f$, and also denote
$$
\xi_k(x,y) = \frac{|u_k(x)-u_k(y)|^{p-2}(u_k(x)-u_k(y))}{|x-y|^{\frac{n+sp}{p'}}}.
$$
Then, $\xi_k\in L^{p'}(\R^n\times\R^n)$ and
$$
\|\xi_k\|_{L^{p'}(\R^n\times\R^n)}^{p'} = [u_k]_{s,p}^p.
$$
Therefore, from Lemma \ref{cota.solucion}, we get that $\{\xi_k\}_{k\in\N}$ is bounded in $L^{p'}(\R^n\times\R^n)$. So, up to some subsequence, there exists a function $\xi \in L^{p'}(\R^n\times \R^n)$ such that
$$
\xi_k \rightharpoonup \xi \text{ weakly in }L^{p'}(\R^n\times \R^n).
$$
Therefore,
\begin{equation}\label{convxi}
\begin{split}
\lim_{k\to\ito}\langle (-\Delta_p)^s u_k, \phi\rangle &= \lim_{k\to\ito}\iint_{\R^n\times\R^n}\xi_k(x,y) \frac{\phi(x)-\phi(y)}{|x-y|^{\frac{n}{p}+s}}\,dx\,dy\\
&= \iint_{\R^n\times\R^n}\xi(x,y)\frac{\phi(x)-\phi(y)}{|x-y|^{\frac{n+sp}{p}}}\,dx\,dy,
\end{split}
\end{equation}
for all $\phi\in W^{s,p}(\R^n)$. In particular, \eqref{convxi} holds for every $\phi\in C_c^\ito(\Omega)$. 
Moreover, by the compactness of the immersion $W^{s,p}_0(D)\subset L^p(D)$ (see \cite{DiNezza-Palatucci-Valdinoci}), since $u_k \rightharpoonup u^*$ weakly in $W^{s,p}_0(D)$ we can conclude that $u_k\to u^*$ a.e. in $\R^n$, then
$$
\xi_k(x,y)\to \frac{|u^*(x)-u^*(y)|^{p-2}(u^*(x)-u^*(y))}{|x-y|^{\frac{n+sp}{p'}}},
$$
a.e. in $\R^n\times\R^n$, from where it follows that 
\begin{equation}\label{xi}
\xi(x,y) = \frac{|u^*(x)-u^*(y)|^{p-2}(u^*(x)-u^*(y))}{|x-y|^{\frac{n+sp}{p'}}}.
\end{equation}

Finally, observe that if $\phi\in C_c^\ito(\Omega)$ then $\phi\in C_c^\ito(\Omega_k)$ for every $k$ sufficiently large (Proposition \ref{convH}). Therefore, from \eqref{convxi} we conclude that
$$
\iint_{\R^n\times\R^n}\xi(x,y)\frac{\phi(x)-\phi(y)}{|x-y|^{\frac{n}{p}+s}}\,dx\,dy=\langle f, \phi \rangle.
$$
The proof is then completed by combining this last equality with \eqref{xi}.
\end{proof}

\begin{rk}
In order to show that $u^* = u_\Omega^f$, what remains is to show that $u^* = 0$ on $\Omega^c$. This is the hard part and is where some geometric hypotheses on the nature of the convergence of the domains needs to be made.
\end{rk}

\begin{teo}\label{teo.main}
Assume that the hypotheses of Theorem \ref{u*sol} are satisfied. If, in addition, 
\begin{equation}\label{condicion.cap}
\cp_{s,p}(\Omega_k \setminus \Omega,D)\to 0,
\end{equation} 
then $u_{\Omega_k}^f\rightharpoonup u_{\Omega}^f$ weakly in $W^{s,p}_0(D)$.
\end{teo}

\begin{proof}
As before, we denote $u_k = u_{\Omega_k}^f$. By Corollary \ref{weakconvergence}, $\{u_k\}_{k\in\N}$ is bounded in $W^{s,p}_{0}(D)$ and therefore we can assume that $u_k\rightharpoonup u^*$ weakly in $W^{s,p}_{0}(D)$.

By Theorem \ref{u*sol} the proof will be finished if we can prove that $u^* \in W^{s,p}_0(\Omega)$, and by Theorem \ref{teocaracterizacion}, it is enough to prove that  $\tilde{u}^* = 0$ \qe\ in $\Omega^c$. 

Consider $\tilde{\Omega}_j = \bigcup_{k\geq j} \Omega_k$ and $E = \bigcap_{j\geq 1} \tilde{\Omega}_j$.

Since  $u_k\rightharpoonup u^*$ in $W^{s,p}_0(D)$, by Mazur's Lemma (see for instance \cite{E-T}), there is a sequence $v_j = \sum_{k=j}^{N_j} a_k^j u_k$ such that $a_k^j \geq 0$, $\sum_{k=j}^{N_j} a_k^j = 1$ and $v_j \to u^*$ strongly in $W^{s,p}_0(D)$.

Since $u_k\in W^{s,p}_0(\Omega_k)$, by Theorem \ref{teocaracterizacion}, $\tilde{u}_k = 0$ \qe\ in $\Omega_k^c$. Therefore, $\tilde{v}_j=\sum_{k=j}^{N_j} a_k^j \tilde{u}_k = 0$ \qe\ in $\cap_{k=j}^{N_j} \Omega_k^c \supset \tilde{\Omega}_j^c$ for every $j\in \N$. 

Then, $\tilde{v}_j = 0$ \qe\ in $\tilde{\Omega}_j^c$ for every $j\in\N$ and, since $\tilde{\Omega}_j^c\subset \tilde{\Omega}_{j+1}^c$, we conclude that   $\tilde{v}_j = 0$ \qe\ $\tilde{\Omega}_i^c$ for every $i\le j$.

On the other hand, since $v_j \to u^*$ strongly in $W^{s,p}_0(D)$, by Proposition \ref{propsubsequence}, $\tilde{v}_{j_k}\rightarrow \tilde{u}^*$ \qe\ Then we conclude that $\tilde{u}^* = 0$ \qe\ in $E^c$.

In order to finish the proof of the theorem, we show that the capacitary condition \eqref{condicion.cap} implies that $\Omega^c\subset E^c$ up to some set of zero $(s,p)-$capacity.

In fact, since  $\cp_{s,p}(\Omega_k \setminus \Omega)\to 0$, passing to a subsequence, if necessary, we can assume that $\cp_{s,p}(\Omega_k \setminus \Omega)\leq \frac{1}{2^k}$. Therefore,
\begin{align*}
\cp_{s,p}(\tilde{\Omega}_j \setminus \Omega) = \cp_{s,p}(\cup_{k\geq j}\Omega_k \setminus \Omega) \le \sum_{k\ge j}\cp_{s,p}(\Omega_k \setminus \Omega) \le \sum_{k\geq j}\frac{1}{2^k}=\frac{1}{2^{j-1}}.
\end{align*}
Recall now that $E\subset \tilde{\Omega}_j$ for every $j\in\N$, then we have that 
$$
\cp_{s,p}(E \setminus \Omega)\le \cp_{s,p}(\tilde{\Omega}_j \setminus \Omega)\le \frac{1}{2^{j-1}} \text{ for every } j\in\N. 
$$
Taking the limit $j\to\infty$, we have that $\cp_{s,p}(E \setminus \Omega) = \cp_{s,p}(\Omega^c \setminus E^c)=0$ and the proof is finished.
\end{proof}

As a simple corollary, we can  show that the convergence of the solutions in Theorem \ref{teo.main} is actually strong.
\begin{coro}\label{fuerte}
Under the assumptions of Theorem \ref{teo.main} we have that $u_{\Omega_m}^f\to u_{\Omega}^f$ strongly in $W^{s,p}_0(D)$.
\end{coro}

\begin{proof}
The proof is simple. Just observe that from the weak convergence $u_{\Omega_m}^f\rightharpoonup u_\Omega^f$ given by Theorem \ref{teo.main}, we get
$$
[u_{\Omega_m}^f]_{s,p}^p = \langle f, u_{\Omega_m}^f\rangle \to \langle f, u_\Omega^f\rangle = [u_\Omega^f]_{s,p}^p.
$$
Since $W^{s,p}_0(D)$ is a uniformly convex Banach space, the result follows.
\end{proof}

\section*{Acknowledgements}
This paper was partially supported by grants UBACyT 20020130100283BA, CONICET PIP 11220150100032CO and ANPCyT PICT 2012-0153. 

'

\def\ocirc#1{\ifmmode\setbox0=\hbox{$#1$}\dimen0=\ht0 \advance\dimen0
  by1pt\rlap{\hbox to\wd0{\hss\raise\dimen0
  \hbox{\hskip.2em$\scriptscriptstyle\circ$}\hss}}#1\else {\accent"17 #1}\fi}
  \def\ocirc#1{\ifmmode\setbox0=\hbox{$#1$}\dimen0=\ht0 \advance\dimen0
  by1pt\rlap{\hbox to\wd0{\hss\raise\dimen0
  \hbox{\hskip.2em$\scriptscriptstyle\circ$}\hss}}#1\else {\accent"17 #1}\fi}
\providecommand{\bysame}{\leavevmode\hbox to3em{\hrulefill}\thinspace}
\providecommand{\MR}{\relax\ifhmode\unskip\space\fi MR }
\providecommand{\MRhref}[2]{%
  \href{http://www.ams.org/mathscinet-getitem?mr=#1}{#2}
}
\providecommand{\href}[2]{#2}
\bibliographystyle{plain}
\bibliography{biblio}

\end{document}